\documentclass[11pt]{article}

\usepackage{mathrsfs}
\usepackage{amsfonts}
\usepackage{amssymb}
\usepackage{amsmath}
\usepackage{color}
\usepackage{multicol}
\usepackage{dsfont}

\usepackage{algorithmic}
\usepackage{array}
\usepackage[caption=false,font=footnotesize,labelfont=rm,textfont=rm]{subfig}
\usepackage{textcomp}
\usepackage{url}
\usepackage{verbatim}
\usepackage{graphicx}
\graphicspath{{figures/}} 
\usepackage[table]{xcolor}
\usepackage{booktabs}
\usepackage{epstopdf}
\usepackage{tabularx}

\newtheorem{theorem}{Theorem}[section]

\numberwithin{equation}{section}

\newtheorem{remark}[theorem]{Remark}
\newenvironment{proof}[1][Proof]{\textbf{#1.}}{\ \rule{0.5em}{0.5em}}

\begin{document}
	\parindent 9mm
	
	\title{Stabilization for a Cascaded ODE-Wave Equation with Boundary Nonlinear Disturbances \thanks{This work was supported by the Natural Science Foundation of Shaanxi Province, China (grant no. 2023-JC-YB-044), and the Fundamental Research Funds for the Central Universities (grant no. xzy012024070).}}
	
	\author{ Zhan-Dong Mei\thanks{Corresponding author. Email: zhdmei@mail.xjtu.edu.cn }, Lan-Xi Tang
		\\{\it  \small  School of Mathematics and Statistics, Xi'an Jiaotong University, Xi'an 710049, China }
	}
	
	\date{}
	\maketitle \thispagestyle{empty}
	
	\begin{abstract}
		In this article, we investigate the problem of exponential stabilization via output feedback for a cascaded system composed of an ordinary differential equation (ODE) and a wave partial differential equation (PDE) under boundary control. Four types of boundary interconnections are considered. In the absence of disturbances, a novel transformation is introduced to incorporate the PDE boundary control input into the ODE subsystem. Based on this transformation, a state feedback controller is designed to achieve exponential stability for both the ODE and PDE components.
		When internal uncertainties and external disturbances that match the control structure are present, a disturbance estimator is constructed. Utilizing this estimator, a Luenberger-type state observer is developed to reject the disturbances and exponentially stabilize the original system via an observer-based control scheme. Furthermore, the boundedness of the entire closed-loop system is rigorously established.

		\vspace{0.5cm} 
		
	\end{abstract}
	
	\section{Introduction}
	~~~~~As is well known, in engineering practice, many physical systems can be modeled as ordinary differential equation (ODE)-partial differential equation (PDE) cascaded systems. These ODE-PDE models play a fundamentally important role in a wide range of practical applications, such as electromagnetic coupling \cite{Krstic2008b}, 3D printing [22], industrial oil-drilling plants \cite{Bekiaris-Liberis2014,Sagert2013,WangJ2020}, cable elevators \cite{WangJ2018}, transport phenomena in gasoline engines \cite{Chebre2010,Jankovic2011}, metal cutting processes \cite{Herty2009}, and traffic flow \cite{Hasan2016,Yu2021}. Consequently, the control of ODE-PDE cascaded systems has attracted considerable attention. However, due to their infinite-dimensional nature, it is challenging to directly apply existing finite-dimensional control theories and techniques to such systems.
	In \cite{Krstic2008b}, the backstepping control approach was applied to an ODE-PDE cascaded system governed by a one-dimensional continuum model, highlighting the potential of PDE backstepping for a variety of physical systems. Since then, this method has been widely adopted to stabilize ODE-PDE cascaded systems--particularly those involving ODEs actuated by first-order hyperbolic PDEs \cite{Di2018,Krstic2009a,Wang2011}, the heat equation \cite{Tang2011,Tang2011a,ZhouZC2015}, the Schr\"{o}dinger equation \cite{Ren2013}, and the wave equation \cite{LiRC2021,Zhen2016,ZhouZC2012}.
	Nevertheless, applying the PDE backstepping approach to an Euler-Bernoulli beam (EBB) remains difficult, except in special cases where the EBB equation with appropriate boundary conditions can be transformed into a Schr?dinger equation \cite{Smyshlyaev2009a}. To address this limitation, Wu and Feng \cite{Wu2022} recently introduced a novel transformation to establish exponential stability for an ODE-EBB system with two interconnections.
	
	The aforementioned ODE-PDE models did not account for external disturbances. However, uncertainties arising from modeling inaccuracies and external environmental factors make the presence of disturbances inevitable. Over the past three decades, significant research attention has been directed toward the stabilization of PDEs under disturbances. Various control strategies have been developed to address such challenges, including the Lyapunov functional approach for the Euler-Bernoulli beam (EBB) equation \cite{Jin2015}, sliding mode control (SMC) for the wave equation \cite{Guo2013a} and the EBB equation \cite{Guo2013b}, as well as active disturbance rejection control (ADRC) applied to the wave equation \cite{Guo2013a} and beam equation \cite{Guo2013b}.
	For ODE-PDE cascaded systems subject to disturbances, the combination of SMC and PDE backstepping has been successfully implemented in several configurations, such as ODE-heat equations \cite{Wang2015}, ODE-Schr\"{o}dinger equations \cite{Liu2018a}, and ODE-wave equations \cite{Zhou2016a}. Similarly, ADRC integrated with PDE backstepping has been employed for ODE-Schr\"{o}dinger equations \cite{GuoYP2015}, ODE-heat equations \cite{Guo2015a}, and ODE-wave equations \cite{Liu2017}. Nevertheless, these ADRC designs predominantly rely on finite-dimensional extended state observers (ESOs), which necessitate high gains and assume bounded disturbances.
	To overcome these limitations, Feng and Guo \cite{Feng2017a} proposed a novel ADRC framework by introducing an infinite-dimensional disturbance estimator in place of a finite-dimensional ESO, achieving stabilization for an anti-stable wave equation. Following a similar approach, subsequent studies have addressed the stabilization of PDEs and ODE-PDE cascaded systems, including wave equations \cite{Mei2020a,Zhou2018b}, beam equations \cite{Mei2020b,Zhou2018a}, ODE-Schr\"{o}dinger \cite{Jia2020a}, ODE-heat equations \cite{Jia2018a,Liu2021}, and ODE-wave equation \cite{Mei2023a}.

	In this paper, we address the problem of exponential stabilization via output feedback for a cascaded system composed of an ODE and a wave PDE under boundary control, subject to matched internal uncertainties and external disturbances. The system is described as follows:
	\begin{equation} \label{beem}
		\left\{\begin{array}{l}
			\dot{X}(t)=AX(t)+B_1w(0,t)+B_2w_t(0,t)\\
			+B_3w(1,t)+B_4w_t(1,t), \; t>0, \\
			w_{tt}(x,t)=w_{xx}(x,t),\;\; x\in (0,1), \; t>0, \\
			w_x(0,t)=0, \; \;   t\ge 0,\\
			w_{x}(1,t)=u(t)+f(w(\cdot,t),w_t(\cdot,t))+d(t), \;\;  t\ge 0, \\
			y(t)=\{w_t(0,t),w(1,t),CX(t)\}, \; \;  t\ge 0,\\
		\end{array}\right.
	\end{equation}	
	where $X\in \mathds{R}^{n\times 1}$ and $w\in L^2(0,1)$  denote the states of the ODE and the wave equation, respectively, $A\in \mathds{R}^{n\times n}$,
	$u(t)$ is the control input applied at $x=1$,
	$B_j\in \mathds{R}^{n\times 1}, j=1,2,3,4$ represent the coupling coefficients, $C$ is the output matrix for the ODE,
	$f(w(\cdot,t),w_t(\cdot,t))$ is nonlinear interior uncertainty and $d(t)$ is the external disturbance, $y(t)$ is the output signal (measurement).
	For notational convenience, the total disturbance is denoted as $F(t):=f(w(\cdot,t),w_t(\cdot,t))+ d(t)$. Throughout this work, we assume that the pair
	$(A,C)$ is observable and $H$ is a matrix such that $A+HC$ is a Hurwitz matrix.

	The dynamic stabilization of system (\ref{beem}) with only one interconnection (i.e., when $B_2=B_3=B_4=0$) has been studied in \cite{Liu2017,Zhou2016a}. In those works, the authors applied two PDE backstepping transformations and designed a sliding mode observer for state estimation. It should be emphasized that, even in the absence of disturbance ($F(t)=0$), the stabilization problem for the case of four interconnections has not been previously addressed. In this paper, we further aim to achieve both disturbance rejection and stabilization of the original system under disturbance. The main contributions of this work are summarized as follows:
	
	1) A novel transformation is proposed to incorporate the wave PDE boundary control input into the ODE subsystem, thereby enabling simultaneous stabilization of both the ODE and PDE parts. This approach significantly differs from the Volterra transformation used in PDE backstepping methods reported in \cite{Liu2017,Zhou2016a}.
	
	2) In contrast to existing studies such as \cite{LiRC2021,Liu2017,Mei2023a,Zhen2016,Zhou2016a}, which only consider a single interconnection, our work addresses the more challenging scenario involving four interconnections.
	
	3) The system under consideration (\ref{beem}) incorporates both internal uncertainties and external disturbances, whereas the systems studied in \cite{Liu2017,Zhou2016a} assumed $B_2=B_3=B_4=0$ and only considered external disturbances. It is worth noting that internal uncertainties also play a critical role, as highlighted in \cite{Zhou2018a,Zhou2018b}.
	
	4) The methods in \cite{Liu2017} rely on state feedback, high-gain designs, and boundedness assumptions on the derivative of the disturbance. Meanwhile, \cite{Zhou2016a} requires high-gain feedback, bounded disturbance derivatives, and four measurements: $w_t(0,t)$, $w(1,t)$, $w_t(1,t)$, and $CX(t)$. In contrast, our approach avoids high-gain designs, does not assume bounded disturbance derivatives, uses only output feedback, and requires only three measurements: $w_t(0,t)$, $w(1,t)$, and $CX(t)$.

	We proceed as follows. In Section 2, we consider system (\ref{beem}) in the absence of disturbance. A novel transformation is introduced to incorporate the wave PDE boundary control input into the ODE subsystem. Based on this transformation, a state feedback controller is designed to stabilize the PDE and track the total disturbance. In Section 3, for the case where $F(t) \neq 0$, we develop an infinite-dimensional disturbance estimator to reconstruct the total disturbance. Using this estimator, we subsequently construct a state observer and an observer-based output feedback controller. The exponential stability of the closed-loop system and boundedness of all signals are rigorously established. Finally, Section 4 provides numerical simulations to illustrate the theoretical results.

	\section{Stabilization in absence of disturbance ($F(t)=0$)}

~~~~This section addresses the problem of state feedback exponential stabilization for system (\ref{beem}) in the absence of disturbance, i.e., when $F(t) = 0$.
The analysis is conducted within the Hilbert state space $\mathbb{H}_1 = H^1(0,1) \times L^2(0,1)$, equipped with its inner product-induced norm described as follows
\begin{align*}
	\|(f,g)\|^2_{\mathbb{H}_1}=\int_0^1[|f'(x)|^2+|g(x)|^2]dx+\alpha |f(1)|^2,\forall \; (f,g)\in \mathbb{H}_1,
\end{align*}
where $\alpha > 0$ is a constant. To this end, we propose a new transformation for the ODE-wave coupled system as follows:
\begin{align}\label{IjiaP}
	\big(Y(t),w(\cdot,t),w_t(\cdot,t)\big)=(I+\mathbb{P})\big(X(t),w(\cdot,t),w_t(\cdot,t)\big),
\end{align}
where $I$ is the identity operator on $\mathds{R}^n\times \mathbb{H}_1$,
\begin{align}\label{PP}
	\nonumber &\mathbb{P}\big(X(t),w(\cdot,t),w_t(\cdot,t)\big)=\bigg(L_3w(0,t)+L_4w(1,t)\\
	&+\int_0^1 L_1(x)w(x,t)dx+\int_0^1L_2(x)w_t(x,t)dx,0,0\bigg),
\end{align}
where $L_1,L_2:[0,1]\rightarrow \mathds{R}^n$ are vector functions and $L_3,L_4$ are vectors to be determined ($(L_1,L_2\in L^2(0,1;\mathds{R}^n))$.
Observe that $\sqrt{\alpha}|w(1,t)|\leq \big\|\big(X(t),w(\cdot,t),w_t(\cdot,t)\big)\big\|_{\mathds{R}^n\times \mathbb{H}_1}$,
\begin{align*}
	&|w(0,t)|\leq (1/\sqrt{\alpha}+1)\big\|\big(w(\cdot,t),w_t(\cdot,t)\big)\big\|_{\mathbb{H}_1}\\
	&\leq (1/\sqrt{\alpha}+1)\big\|\big(X(t),w(\cdot,t),w_t(\cdot,t)\big)\big\|_{\mathds{R}^n\times \mathbb{H}_1},\\
	&\left\|\int_0^1 L_1(x)w(x,t)dx\right\|_{\mathds{R}^n}\\
	&\leq \left(\int_0^1 \|L_1(x)\|^2_{\mathds{R}^n}dx\right)^{1/2}\left(\int_0^1|w(x,t)|^2dx\right)\\
	&\leq \left(\int_0^1 \|L_1(x)\|^2_{\mathds{R}^n}dx\right)^{1/2}\big\|\big(w(\cdot,t),w_t(\cdot,t)\big)\big\|_{\mathbb{H}_1},\\
	&\left\|\int_0^1 L_2(x)w_t(x,t)dx\right\|_{\mathds{R}^n}\\
	&\leq \left(\int_0^1 \|L_2(x)\|^2_{\mathds{R}^n}dx\right)^{1/2}\left(\int_0^1|w_t(x,t)|^2dx\right)\\
	&\leq  \left(\int_0^1 \|L_2(x)\|^2_{\mathds{R}^n}dx\right)^{1/2}\big\|\big(w(\cdot,t),w_t(\cdot,t)\big)\big\|_{\mathbb{H}_1}.
\end{align*}
Observe that $\mathbb{P}$ is a bounded linear operator on $\mathds{R}^n \times \mathbb{H}_1$. Therefore, the identities $(I + \mathbb{P})(I - \mathbb{P}) = (I - \mathbb{P})(I + \mathbb{P}) = I - \mathbb{P}^2 = I$ imply that $I + \mathbb{P}$ is invertible. Consequently, its inverse is given by $I - \mathbb{P}$. Moreover,
\begin{align}\label{XdaoY0}
	\nonumber&Y(t)=X(t)+L_3w(0,t)+L_4w(1,t)+\int_0^1 L_1(x)w(x,t)dx\\
	&+\int_0^1L_2(x)w_t(x,t)dx,
\end{align}
and there holds
\begin{align*}
	&\dot{Y}(t)=\dot{X}(t)+L_3w_t(0,t)+L_4w_t(1,t)\\
	&+\int_0^1 L_1(x)w_t(x,t)dx+\int_0^1L_2(x)w_{tt}(x,t)dx\\
	&=AX(t)+B_1w(0,t)+B_2w_t(0,t)+B_3w(1,t)\\
	&+B_4w_t(1,t)+L_3w_t(0,t)+L_4w_t(1,t)\\
	&+\int_0^1 L_1(x)w_t(x,t)dx+\int_0^1L_2(x)w_{xx}(x,t)dx\\
	&=AY(t)-AL_3w(0,t)-AL_4w(1,t)+B_1w(0,t)\\
	&-\int_0^1 AL_1(x)w(x,t)dx-\int_0^1AL_2(x)w_t(x,t)dx\\
	&+B_2w_t(0,t)+B_3w(1,t)+B_4w_t(1,t)+L_3w_t(0,t)\\
	&+L_4w_t(1,t)+\int_0^1 L_1(x)w_t(x,t)dx+L_2(x)w_x(x,t)|_{x=0}^1\\
	&-L'_2(x)w(x,t)|_{x=0}^1+\int_0^1L''_2(x)w(x,t)dx\\
	&=AY(t)+(B_1-AL_3+L'_2(0))w(0,t)+(B_3-AL_4\\
	&-L'_2(1))w(1,t)+(B_2+L_3)w_t(0,t)+(B_4+L_4)w_t(1,t)\\
	&+\int_0^1 [L''_2(x)-AL_1(x)]w(x,t)dx+L_2(1)u(t)\\
	&+\int_0^1 [L_1(x)-AL_2(x)]w_t(x,t)dx.
\end{align*}
We construct a controller
\begin{align}\label{controllernodisturbance}
	\nonumber&u(t)=-\alpha w(1,t)-\beta w_t(1,t)+K^T\bigg[X(t)+L_3w(0,t)\\
	&+L_4w(1,t)+\int_0^1 L_1(x)w(x,t)dx+\int_0^1L_2(x)w_t(x,t)dx\bigg],
\end{align}
where $K$ is a column vector whose value is to be determined. This results in the closed-loop system:
\begin{equation} \label{closedloopnodisturbance}
	\left\{\begin{array}{l}
		\dot{X}(t)=AX(t)+B_1w(0,t)+B_2w_t(0,t)\\
		+B_3w(1,t)+B_4w_t(1,t), \; t>0, \\
		w_{tt}(x,t)=w_{xx}(x,t),\;\; x\in (0,1), \; t>0, \\
		w_x(0,t)=0, \; \;   t\ge 0,\\
		w_{x}(1,t)=-\alpha w(1,t)-\beta w_t(1,t)\\
		+K^T\bigg[X(t)+L_3w(0,t)
		+L_4w(1,t)\\
		+\int_0^1 L_1(x)w(x,t)dx+\int_0^1L_2(x)w_t(x,t)dx\bigg], \; \;  t\ge 0.\\
	\end{array}\right.
\end{equation}
Combine (\ref{XdaoY0}) and (\ref{closedloopnodisturbance}) to derive
\begin{align}\label{Yw}
	\left\{
	\begin{array}{ll}
		\dot{Y}(t)=[A+L_2(1)K^T]Y(t)+(B_1-AL_3\\
		+L'_2(0))w(0,t)+(B_2+L_3)w_t(0,t)
		+(B_3-AL_4\\
		-L'_2(1)-\alpha L_2(1))w(1,t)+(B_4+L_4\\
		-\beta L_2(1))w_t(1,t)
		+\int_0^1 [L''_2(x)-AL_1(x)]w(x,t)dx\\
		+\int_0^1 [L_1(x)-AL_2(x)]w_t(x,t)dx,\\
		w_{tt}(x,t)=w_{xx}(x,t),\\
		w_x(0,t)=0,\\
		w_x(1,t)=-\alpha w(1,t)-\beta w_t(1,t)+K^TY(t).
	\end{array}
	\right.
\end{align}
Our goal is to establish the exponential stability of system (\ref{Yw}). We therefore make the following assumption on the parameters $L_1, L_2, L_3, L_4$:
\begin{align}\label{L12340}
	\left\{
	\begin{array}{ll}
		L''_2(x)=A^2L_2(x), L_1(x)=AL_2(x),\\
		B_1-AL_3+L'_2(0)=0,\\
		B_2+L_3=0,      B_3-AL_4-L'_2(1)-\alpha L_2(1)=0,\\
		B_4+L_4-\beta L_2(1)=0,
	\end{array}
	\right.
\end{align}
which tells us that $L_1$, $L_3$ and $L_4$ are completely determined provided $L_2(x)$ is derived.
The closed-loop system then becomes
\begin{align}\label{clst}
	\left\{
	\begin{array}{ll}
		\dot{Y}(t)=[A+L_2(1)K^T]Y(t),\\
		w_{tt}(x,t)=w_{xx}(x,t),\\
		w_x(0,t)=0,\\
		w_x(1,t)=-\alpha w(1,t)-\beta w_t(1,t)+K^TY(t),
	\end{array}
	\right.
\end{align}
which is abstractly written by
\begin{align}\label{clst}
	\left\{
	\begin{array}{ll}
		\dot{Y}(t)=[A+L_2(1)K^T]Y(t),\\
		\frac{d}{dt}\left(
		\begin{array}{c}
			w(\cdot,t) \\
			w_t(\cdot,t) \\
		\end{array}
		\right)=\mathbb{A}_1\left(
		\begin{array}{c}
			w(\cdot,t) \\
			w_t(\cdot,t) \\
		\end{array}
		\right)
		+\mathbb{B}_{11}[K^TY(t)].
	\end{array}
	\right.
\end{align}
where the operator $\mathbb{A}_1$ is defined by
\begin{align*}
	\left\{
	\begin{array}{ll}
		\mathbb{A}_1(f,g)=(g,f''),(f,g)\in D(\mathbb{A}_1), \\
		D(\mathbb{A}_1)=\{(f,g)\in H^2(0,1)\times H^1(0,1)| f'(0)=0,\\
		f'(1)=-\alpha f(1)-\beta g(1),
	\end{array}
	\right.
\end{align*}
$\mathbb{B}_{11}=\delta(x-1)$ with $\delta$ being the Dirac function.
It is well-known that $\mathbb{A}_1$ generates an exponentially stable $C_0$-semigroup $e^{\mathbb{A}_1t}$ on $\mathbb{H}_1$.

The rest of this section is devoted to two main objectives: the solvability of ODE (\ref{L12340}) and the exponential stability of the closed-loop system (\ref{closedloopnodisturbance}). We begin by introducing the analytic function $\mathcal{G}(z)$, defined as
\begin{align*}
	\mathcal{G}(z)=\left\{
	\begin{array}{ll}
		\frac{\sinh(z)}{z},\;\; z\neq 0, z\in \mathds{C}, \\
		1, z=0.
	\end{array}
	\right.
\end{align*}
This function has been employed in studying the output tracking problem for a one-dimensional wave equation \cite{Feng2020a}, and it satisfies the identity $A\mathcal{G}(A) = \sinh A$.
\begin{theorem}\label{iandii}
	Assume that $\sigma(A)\bigcap \sigma(\mathbb{A}_1)=\varnothing$.\\
	(i) The matrix $A\sinh A+(\alpha+\beta A)\cosh A$ is invertible.\\
	(ii) The equation (\ref{L12340}) has a solution given by
	\begin{align}\label{L1234zhi}
		\left\{
		\begin{array}{ll}
			L_2(x)=-x\mathcal{G}(Ax)(AB_2+B_1)+ (\cosh Ax)Q, \\
			L_1(x)=-(\sinh (Ax))(AB_2+B_1)+A(\cosh Ax)Q,\\
			L_3=-B_2,\\
			L_4=-\beta \mathcal{G}(A)(AB_2+B_1)+ \beta(\cosh A)Q-B_4,
		\end{array}
		\right.
	\end{align}
	where
	\begin{align*}
		&Q=[A\sinh A+(\alpha+\beta A)\cosh A]^{-1}\{AB_4+B_3\\
		&+[\cosh A+(\alpha+\beta A)\mathcal{G}(A)][AB_2+B_1]\}.
	\end{align*}
	(iii) If, in addition, the pair $(A, A B_2 + B_1 + (\cosh A)(A B_4 + B_3))$ is controllable, then $K$ can be chosen so that $A + L_2(1)K^T$ is Hurwitz, provided that the closed-loop system (\ref{clst}) is exponentially stable.
\end{theorem}
\begin{proof}\ \
	(i) It can be readily established that $(\mathbb{A}_1)^{-1}$ exists and is compact. Consequently, the spectrum of $\mathbb{A}_1$ is discrete and consists entirely of eigenvalues. For $(f, g) \in D(\mathbb{A}_1)$, we consider the characteristic equation $\mathbb{A}_1(f, g) = \lambda (f, g)$, which reads:
	\begin{align*}
		\left\{
		\begin{array}{ll}
			f''(x)=\lambda^2 f(x),x\in(0,1),\\
			f'(0)=0,f'(1)=-(\alpha+\beta\lambda)f(1).
		\end{array}
		\right.
	\end{align*}
	It follows that $\lambda \in \sigma(\mathbb{A}_1)$ is equivalent to the following condition:
	\begin{align*}
		\lambda\sinh\lambda+(\alpha+\beta\lambda)\cosh\lambda=0.
	\end{align*}
	This implies that for any $\mu \in \rho(A)$, the expression $\mu \sinh \mu + (\alpha + \beta \mu) \cosh \mu$ is nonzero. Consequently, the operator $A \sinh A + (\alpha + \beta A) \cosh A$ is thus invertible on $\mathds{R}^n$.
	
	(ii) From (\ref{L12340}), we obtain $L_1(x) = A L_2(x)$, $L_3 = -B_2$, $L_4 = \beta L_2(1) - B_4$, and
	\begin{align}\label{L12}
		\left\{
		\begin{array}{ll}
			L''_2(x)=A^2L_2(x),  \\
			L'_2(0)=-AB_2-B_1,\\
			L'_2(1)+(\alpha+\beta A)L_2(1)=AB_4+B_3.
		\end{array}
		\right.
	\end{align}
	By the standard existence theorem for ordinary differential equations, the solution to (\ref{L12}) exists and can be expressed in the form
	\begin{align*}
		L_2(x)=x\mathcal{G}(Ax)P+(\cosh(Ax))Q_0,
	\end{align*}
	where $P$ and $Q_0$ are vectors whose values are to be determined. The boundary conditions of (\ref{L12}) provide the constraints:
	\begin{align*}
		\left\{
		\begin{array}{ll}
			P=-AB_2-B_1, \\
			AB_4+B_3=-(\cosh A)P+A(\sinh A)Q_0\\
			+(\alpha+\beta A)[\mathcal{G}(A)P+(\cosh A)Q_0].
		\end{array}
		\right.
	\end{align*}
	Since by (i) the matrix $A\sinh A+(\alpha+\beta A)\cosh A$ is invertible, we derive
	\begin{align*}
		L_2(x)=-[AB_2+B_1]x\mathcal{G}(Ax)+(\cosh(Ax))Q,
	\end{align*}
	where
	\begin{align*}
		&Q=[A\sinh A+(\alpha+\beta A)\cosh A]^{-1}\{AB_4+B_3\\
		&+[\cosh A+(\alpha+\beta A)\mathcal{G}(A)][AB_2+B_1]\}.
	\end{align*}
	
	(iii) We compute
	\begin{align*}
		&L_2(1)=-\mathcal{G}(A)(AB_2+B_1)+ (\cosh A)Q\\
		&=[A\sinh A+(\alpha+\beta A)\cosh A]^{-1}\{-\mathcal{G}(A)[A\sinh A\\
		&+(\alpha+\beta A)\cosh A](AB_2+B_1)+(\cosh A)(AB_4+B_3)\\
		&+(\cosh A)^2(AB_2+B_1)+(\alpha\\
		&+\beta A)\cosh A\mathcal{G}(A)(AB_2+B_1)\}\\
		&=[A\sinh A+(\alpha+\beta A)\cosh A]^{-1}[AB_2+B_1\\
		&+\cosh A (AB_4+B_3)],
	\end{align*}
	where we have used the identity $A\mathcal{G}(A) = \sinh A$.
	
	Since $(A,AB_2+B_1+\cosh A (AB_4+B_3))$ is controllable, by the rank criterion, it follows that
	${\rm rank}\big(AB_2+B_1+\cosh A (AB_4+B_3),A(AB_2+B_1+\cosh A (AB_4+B_3)),\cdots, A^{n-1}(AB_2+B_1+\cosh A (AB_4+B_3))\big)=n$.
	This indicates that ${\rm rank}\big(W(AB_2+B_1+\cosh A (AB_4+B_3)),WA(AB_2+B_1+\cosh A (AB_4+B_3)),\cdots, WA^{n-1}(AB_2+B_1+\cosh A (AB_4+B_3))\big)=n$,
	where $W=[A\sinh A+(\alpha+\beta A)\cosh A]^{-1}$. One can see that $W$ and $A$ are interchangeable,
	thereby ${\rm rank}\big(W(AB_2+B_1+\cosh A (AB_4+B_3)),AW(AB_2+B_1+\cosh A (AB_4+B_3)),\cdots,A^{n-1}W(AB_2$
	$+B_1+\cosh A (AB_4+B_3))\big)=n$.
	This tells us that $\big(A,W(AB_2+B_1+\cosh A (AB_4+B_3))\big)$, or $(A,L_2(1))$
	is controllable. The controllability of $(A,L_2(1))$ indicates that
	$K$ can be chosen such that $A+L_2(1)K^T$ is a Hurwitz matrix.
	This completes the proof.
\end{proof}

{\bf Throughout the rest of this paper, we assume that $\sigma(A)\bigcap \sigma(\mathbb{A}_1)=\varnothing$, and that the pair
	$(A,AB_2+B_1+(\cosh A)(AB_4+B_3))$  is controllable, ensuring that $A+L_2(1)K^T$ is a Hurwitz matrix.}

\begin{theorem}\label{state1}
	The closed-loop system (\ref{closedloopnodisturbance}) is exponentially stable.
\end{theorem}
\begin{proof}\ \
	Since $A+L_2(1)K^\mathrm{T}$ is a Hurwitz matrix, the solution $Y(t) = e^{(A+L_2(1)K^\mathrm{T})t}Y(0)$ is exponentially stable. This further implies that $K^\mathrm{T}Y(t)$ is also exponentially stable. By invoking \cite{Guo2008}, it follows that the semigroup $e^{\mathbb{A}_1 t}$ is exponentially stable and the operator $\mathbb{B}_{11}$ is admissible for $e^{\mathbb{A}_1 t}$. Then, applying \cite[Lemma 2.1]{Zhou2018b}, we conclude that the state $(w(\cdot,t), w_t(\cdot,t))$ is exponentially stable. Consequently, the composite state $(X(t), w(\cdot,t), w_t(\cdot,t)) = (I - \mathbb{P})(Y(t), w(\cdot,t), w_t(\cdot,t))$ is exponentially stable.
\end{proof}

\begin{remark}\rm
	Theorem \ref{state1} implies that, in contrast to the existing literature \cite{LiRC2021,Liu2017,Zhen2016,Zhou2016a} which deals with systems having only a single interconnection from the PDE to the ODE, our controller design successfully achieves exponential stabilization for an ODE-wave system with four interconnections. This notable result is made possible precisely by the novel transformation (\ref{IjiaP}).
\end{remark}
\begin{remark}\rm
	In the particular case where $B_2 = B_3 = B_4 = 0$, the assumed controllability of the pair $(A, A B_2 + B_1 + (\cosh A)(A B_4 + B_3))$ reduces to that of $(A, B_1)$, thereby recovering the standard controllability condition used in \cite{Zhou2016a}. Under this condition, the controller (\ref{controllernodisturbance}) becomes:
	\begin{align}\label{B234eq0}
		\nonumber&u(t)=-\alpha w(1,t)-\beta w_t(1,t)+K^T\bigg[X(t)+L_4w(1,t)\\
		&+\int_0^1 L_1(x)w(x,t)dx+\int_0^1L_2(x)w_t(x,t)dx\bigg],
	\end{align}
	where
	\begin{align*}
		\left\{
		\begin{array}{ll}
			L_2(x)=-x\mathcal{G}(Ax)B_1+ (\cosh Ax)Q_1, \\
			L_1(x)=-(\sinh (Ax))B_1+A(\cosh Ax)Q_1,\\
			L_4=-\beta \mathcal{G}(A)B_1+ \beta(\cosh A)Q_1,
		\end{array}
		\right.
	\end{align*}
	$Q_1=\left[A\sinh A+(\alpha+\beta A)\cosh A\right]^{-1}\left[\cosh A+(\alpha+\beta A)\mathcal{G}(A)\right]B_1].$
	Our method achieves stabilization with only one invertible transformation (\ref{IjiaP}), thereby simplifying the control architecture.
	This offers a notable advantage over \cite{Zhou2016a}, which relies on two separate PDE backstepping transformations, and yields a simpler controller (\ref{controllernodisturbance}).
\end{remark}

	\section{Stabilization in presence of disturbance ($F(t)\neq0$)}
	
	This section considers the output feedback exponential stabilization of system (\ref{beem}) under nonlinear internal uncertainties and external disturbances ($F(t) \neq 0$).
	The existence of the total disturbance necessitates its prior rejection. We therefore begin by designing a disturbance estimator. Based on the wave part of system (\ref{beem}), the following estimator is thus proposed:
	in \cite[(4)]{Wei2020} described by
	\begin{equation} \label{transfer0}
		\left\{\begin{array}{l}
			z_{tt}(x,t)=z_{xx}(x,t),\;\; x\in (0,1), \; t>0, \\
			z_x(0,t)=k [z_t(0,t)-w_t(0,t)], \; \;   t\ge 0,\\
			z_x(1,t)=-\alpha [z(1,t)-w(1,t)]+u(t) \;\;  t\ge 0, \\
			p_{tt}(x,t)=p_{xx}(x,t),\;\; x\in (0,1), \; t>0, \\
			p_x(0,t)=-kp_t(0,t),p(1,t)=z(1,t)-w(1,t) \;\;  t\ge 0, \\
		\end{array}\right.
	\end{equation}
	which is computed solely from the available input and output of (\ref{beem}), where $k > 0$ is a tuning parameter.
	
	Consider the error variables defined by $\widehat{z}(x,t)=z(x,t)-w(x,t)$ and $\widehat{p}(x,t)=p(x,t)-\widehat{z}(x,t)$. The following error system is then derived:
	\begin{equation} \label{perror}
		\left\{\begin{array}{l}
			\widehat{z}_{tt}(x,t)=\widehat{z}_{xx}(x,t),\;\; x\in (0,1), \; t>0, \\
			\widehat{z}_x(0,t)=k\widehat{z}_t(0,t), \widehat{z}_x(1,t)=-\alpha \widehat{z}(1,t)-F(t), \;\;  t\ge 0, \\
			\widehat{p}_{tt}(x,t)=\widehat{p}_{xx}(x,t),\;\; x\in (0,1), \; t>0, \\
			\widehat{p}_{x}(0,t)=k\widehat{p}_t(0,t),\widehat{p}(1,t)=0,\;\;  t\ge 0, \\
		\end{array}\right.
	\end{equation}
	We see that the system is cascaded by the $\widehat{z}$-subsystem and the $\widehat{p}$-subsystem, and the latter is independent of the former. Here, the Hilbert space is defined as $\mathbb{H}_2 = H_E^1(0,1) \times L^2(0,1)$, with $H_E^1(0,1) = { f \in H^1(0,1) \mid f(1) = 0 }$, and the norm is given as follows
	$$\|(f,g)\|^2_{\mathbb{H}_2}=\int_0^1[|f''(x)|^2+|g(x)|^2]dx, \;\forall (f,g)\in \mathbb{H}_2.$$
	We define $\mathbb{A}$ by
	\begin{align*}
		\left\{
		\begin{array}{ll}
			\mathbb{A}(f,g)=(g,f''),(f,g)\in D(\mathbb{A}), \\
			D(\mathbb{A})=\{(f,g)\in H^2(0,1)\times H^1(0,1)| f'(0)=kg(0),\\
			f'(1)=-\alpha f(1)\},
		\end{array}
		\right.
	\end{align*}
	and define $\mathbb{A}_2$ by
	\begin{align*}
		\left\{
		\begin{array}{ll}
			\mathbb{A}_2(f,g)=(g,f''),(f,g)\in D(\mathbb{A}_2), \\
			D(\mathbb{A}_2)=\{(f,g)\in (H^2(0,1)\bigcap H_E^1(0,1))\times \\
			H_E^1(0,1)| f'(0)=kg(0)\}.
		\end{array}
		\right.
	\end{align*}
	Then system (\ref{perror}) can be written abstractly as follows
	\begin{align}\label{perrorsemigroup}
		\left\{
		\begin{array}{ll}
			\frac{d}{dt}\left(
			\begin{array}{c}
				\widehat{z}(\cdot,t)\\
				\widehat{z}_t(\cdot,t) \\
			\end{array}
			\right)=\mathbb{A}\left(
			\begin{array}{c}
				\widehat{z}(\cdot,t)\\
				\widehat{z}_t(\cdot,t) \\
			\end{array}
			\right)-\mathbb{B}F(t), \\
			\frac{d}{dt}\left(
			\begin{array}{c}
				\widehat{p}(\cdot,t)\\
				\widehat{p}_t(\cdot,t) \\
			\end{array}
			\right)=\mathbb{A}_2\left(
			\begin{array}{c}
				\widehat{p}(\cdot,t)\\
				\widehat{p}_t(\cdot,t) \\
			\end{array}
			\right),
		\end{array}
		\right.
	\end{align}
	where $\mathbb{B}=(0,\delta(x-1))$ with $\delta$ being the Dirac delta function. The operators $\mathbb{A}$ and $\mathbb{A}_2$ are well-known to generate $C_0$-semigroups $e^{\mathbb{A}t}$ on $\mathbb{H}1$ and $e^{\mathbb{A}_2t}$ on $\mathbb{H}_2$, respectively. Therefore, there exist constants $M_{\mathbb{A}}, M_{\mathbb{A}2}, \omega_{\mathbb{A}}, \omega_{\mathbb{A}_2} > 0$ such that
	\begin{align}\label{AA2}
		\|e^{\mathbb{A}t}\|_{\mathbb{H}_1}\leq M_{\mathbb{A}}e^{-\omega_{\mathbb{A}}t},
		\|e^{\mathbb{A}_2t}\|_{\mathbb{H}_2}\leq M_{\mathbb{A}_2}e^{-\omega_{\mathbb{A}_2}t},t\geq 0.
	\end{align}
	Moreover, by \cite{Wei2020}, $\mathbb{B}$ is admissible for $e^{\mathbb{A}t}$.
	
	According to \cite[Remark 2.2]{Wei2020}, $F(t) = -\widehat{z}_x(1,t) - \alpha \widehat{z}(1,t) = -p_x(1,t) - \alpha p(1,t) + \widehat{p}_x(1,t)$. This leads to the approximation $F(t) \approx -p_x(1,t) - \alpha p(1,t)$, and hence $-p_x(1,t) - \alpha p(1,t)$ can be regarded as the estimate of the total disturbance $F(t)$.
	
	Armed with the disturbance estimator, we construct a Luenberger observer for state estimation.
	This observer is designed by duplicating the plant dynamics (\ref{beem}), which serves as the nominal model:
	\begin{equation} \label{stateobserver}
		\left\{\begin{array}{l}
			\dot{\widehat{X}}(t)=A\widehat{X}(t)+B_1\widehat{w}(0,t)+B_2w_t(0,t)+B_3w(1,t)\\
			+B_4\widehat{w}_t(1,t)+H[C\widehat{X}(t)-CX(t)],\;\;  t\ge 0, \\
			\widehat{w}_{tt}(x,t)=\widehat{w}_{xx}(x,t),\;\; x\in (0,1), \; t>0, \\
			\widehat{w}_x(0,t)=k [\widehat{w}_t(0,t)-w_t(0,t)], \; \;   t\ge 0,\\
			\widehat{w}_x(1,t)=-\alpha [\widehat{w}(1,t)-w(1,t)]\\
			+u(t)-p_x(1,t)-\alpha p(1,t) \;\;  t\ge 0,
		\end{array}\right.
	\end{equation}
	where $-p_x(1,t)-\alpha p(1,t)$ plays the role of the total disturbance $F(t)$.
	We introduce the observer errors $\widetilde{X}(t) = \widehat{X}(t) - X(t)$ and $\widetilde{z}(x,t) = \widehat{w}(x,t) - w(x,t)$. The resulting error system is given by:
	\begin{equation} \label{stateerror}
		\left\{\begin{array}{l}
			\dot{\widetilde{X}}(t)=(A+HC)\widetilde{X}(t)+B_1\widetilde{w}(0,t)+B_4\widetilde{w}_t(1,t),\;\;  t\ge 0, \\
			\widetilde{w}_{tt}(x,t)=\widetilde{w}_{xx}(x,t),\;\; x\in (0,1), \; t>0, \\
			\widetilde{w}_x(0,t)=k \widetilde{w}_t(0,t), \; \;   t\ge 0,\\
			\widetilde{w}_x(1,t)=-\alpha \widetilde{w}(1,t)-\widehat{p}_x(1,t). \;\;  t\ge 0, \\
			\widehat{p}_{tt}(x,t)=\widehat{p}_{xx}(x,t),\;\; x\in (0,1), \; t>0, \\
			\widehat{p}_{x}(0,t)=k\widehat{p}_t(0,t),\widehat{p}(1,t)=0,\;\;  t\ge 0,
		\end{array}\right.
	\end{equation}
	
	The following Theorem \ref{wwan} establishes the exponential convergence of the observer designed in (\ref{stateobserver}).
	\begin{theorem}\label{wwan}
		For any $(\widetilde{X}(0),\widetilde{w}(\cdot,0),\widetilde{w}_t(\cdot,0),\widehat{p}(\cdot,0),\widehat{p}_t(\cdot,0))\in \mathds{R}^n\times\mathbb{H}_1\times \mathbb{H}_2$, there
		exists a unique solution $(\widetilde{X}(t),\widetilde{X}(t),\widetilde{w}(\cdot,t),\widetilde{w}_t(\cdot,t),\widehat{p}(\cdot,t),\widehat{p}_t(\cdot,t))\in C(0,\infty;\mathds{R}^n\times\mathbb{H}_1\times\mathbb{H}_2)$. Moreover, there exist two constants $M_1$
		and $\gamma_1$ such that
		$\|(\widetilde{X}(t),\widetilde{w}(\cdot,t),\widetilde{w}_t(\cdot,t),\widehat{p}(\cdot,t),\widehat{p}_t(\cdot,$
		$t))\|_{\mathds{R}^n\times\mathbb{H}_1\times \mathbb{H}_2}\leq M_1e^{-\gamma_1t}$, $t\geq 0$.
	\end{theorem}
	\begin{proof}\ \
		Set $\widetilde{\epsilon}(x,t)=\widetilde{w}(x,t)+\widehat{p}(x,t)$ to obtain
		\begin{equation} \label{epsilonwan}
			\left\{\begin{array}{l}
				\widetilde{\epsilon}_{tt}(x,t)=\widetilde{\epsilon}_{xx}(x,t),\;\; x\in (0,1), \; t>0, \\
				\widetilde{\epsilon}_x(0,t)=k \widetilde{\epsilon}_t(0,t), \widetilde{\epsilon}_x(1,t)=-\alpha \widetilde{\epsilon}(1,t), \;\;  t\ge 0. \\
			\end{array}\right.
		\end{equation}
		Then, the solution $(\widetilde{\epsilon}(\cdot,t),\widetilde{\epsilon}_t(\cdot,t)) = e^{\mathbb{A}t}(\widetilde{\epsilon}(\cdot,0),\widetilde{\epsilon}_t(\cdot,0))$ uniquely exists in $C(0,\infty;\mathbb{H}_1)$ and is exponentially stable. Similarly, $(\widehat{p}(\cdot,t),\widehat{p}_t(\cdot,t)) = e^{\mathbb{A}_2t}(\widehat{p}(\cdot,0),\widehat{p}_t(\cdot,0))$ also enjoys unique existence and exponential stability. Consequently, the claimed properties for the $\widetilde{w}$-part and $\widehat{p}$-part are established.
		
		We now compute
		\begin{align}\label{XX}
			\nonumber& |\widetilde{w}(0,t)|^2\leq 2|\widetilde{w}(1,t)|^2+2\int_0^t|\widetilde{w}_x(x,t)|^2dx\\
			&\leq (2/\alpha+2)\|(\widetilde{w}(\cdot,t),\widetilde{w}_t(\cdot,t))\|^2_{\mathbb{H}_1}.
		\end{align}
		This implies that $\widetilde{w}(0,t)$ is exponentially stable. Define the energy function
		$$E_0(t) = \int_0^1 \left[|\widetilde{w}_t(x,t)|^2 + |\widetilde{w}_x(x,t)|^2 + |\widehat{p}_t(x,t)|^2 + |\widehat{p}_x(x,t)|^2\right] dx + \alpha |\widetilde{w}(1,t)|^2$$
		and the auxiliary function $\rho(t) = \int_0^1 x \widetilde{w}_t(x,t) \widetilde{w}_x(x,t) dx$. Then there exist positive constants $M_0$ and $\gamma_0$ such that $|\rho(t)| \leq E_0(t) \leq M_0^2 e^{-2\gamma_0 t}$ and
		\begin{align}\label{Jisuan}
			\nonumber &\dot{\rho}(t)=\int_0^1x\widetilde{w}_{tt}(x,t)\widetilde{w}_x(x,t)dx+\int_0^1x\widetilde{w}_{t}(x,t)\widetilde{w}_{xt}(x,t)dx\\
			\nonumber&=x[|\widetilde{w}_{t}(x,t)|^2+|\widetilde{w}_x(x,t)|^2]|_{x=0}^1\\
			\nonumber&-\int_0^1[|\widetilde{w}_{t}(x,t)|^2+|\widetilde{w}_x(x,t)|^2]dx\\
			\nonumber&=|\widetilde{w}_{t}(1,t)|^2+|\widetilde{w}_x(1,t)|^2\\
			&-\int_0^1[|\widetilde{w}_{t}(x,t)|^2+|\widetilde{w}_x(x,t)|^2]dx.
		\end{align}
		Given $0<\gamma<2\gamma_0$. We use (\ref{Jisuan}) to derive
		\begin{align}\label{Jisuan11}
			\nonumber&\frac{d}{dt}[e^{\gamma t}\rho(t)]=\gamma e^{\gamma t}\rho(t)+e^{\gamma t}\dot{\rho}(t)\\
			\nonumber&=\gamma e^{\gamma t}\rho(t)+e^{\gamma t}\bigg[|\widetilde{w}_{t}(1,t)|^2+|\widetilde{w}_x(1,t)|^2\\
			&-\int_0^1[|\widetilde{w}_{t}(x,t)|^2+|\widetilde{w}_x(x,t)|^2]dx\bigg],
		\end{align}
		which gives
		\begin{align}\label{Jisuan22}
			\nonumber& \int_0^\infty e^{\gamma t}|\widetilde{w}_{t}(1,t)|^2dt\\
			\nonumber&\leq \int_0^\infty e^{\gamma t}[|\widetilde{w}_{t}(1,t)|^2+|\widetilde{w}_x(1,t)|^2]dt\\
			\nonumber&=\int_0^\infty\frac{d}{dt}[e^{\gamma t}\rho(t)]dt-\int_0^\infty\gamma e^{\gamma t}\rho(t)dt\\
			\nonumber&+\int_0^\infty \bigg[e^{\gamma t}\int_0^1[|\widetilde{w}_{t}(x,t)|^2+|\widetilde{w}_x(x,t)|^2]dx\bigg]dt,\\
			\nonumber&\leq -\rho(0)+\int_0^\infty(\gamma+1)M_0^2e^{\gamma t}e^{-2\gamma_0t}dt\\
			&=E(0)+\frac{(1+\gamma)M_0^2}{2\gamma_0-\gamma}.
		\end{align}
		This implies that $e^{\gamma t}\widetilde{w}_{t}(1,t)\in L^2(0,\infty)$. Furthermore, since $\widetilde{w}(0,t)$ is exponentially stable,
		it follows from \cite[Lemma 2.1]{Zhou2018b} that $\widetilde{X}(t)$ is also exponentially stable.
	\end{proof}
	
	The exponential stability of the disturbance-free cascaded ODE-wave system (\ref{beem}) under controller (\ref{controllernodisturbance}),
	combined with the state observer (\ref{stateobserver}), motivates the design of an output-feedback controller that incorporates a disturbance estimator.
	\begin{align}\label{controller}
		\nonumber&u(t)=-\alpha w(1,t)-\beta \widehat{w}_t(1,t)+K^T\bigg[\widehat{X}(t)\\
		\nonumber&+L_3\widehat{w}(0,t)+L_4w(1,t)+\int_0^1 L_1(x)\widehat{w}(x,t)dx\\
		&+\int_0^1L_2(x)\widehat{w}_t(x,t)dx\bigg]+p_x(1,t)+\alpha p(1,t),
	\end{align}
	where the last two terms $p_x(1,t) + \alpha p(1,t)$ function to cancel the total disturbance. Accordingly, the closed-loop system under controller (\ref{controller}) takes the form
	\begin{align} \label{closedloopdisturbance}
		\left\{\begin{array}{l}
			\dot{X}(t)=AX(t)+B_1w(0,t)+B_2w_t(0,t)\\
			+B_3w(1,t)+B_4w_t(1,t), \; t>0, \\
			w_{tt}(x,t)=w_{xx}(x,t),\;\; x\in (0,1), \; t>0, \\
			w_x(0,t)=0, \; \;   t\ge 0,\\
			w_{x}(1,t)=-\alpha w(1,t)-\beta \widehat{w}_t(1,t)+K^T\bigg[\widehat{X}(t)\\
			+L_3\widehat{w}(0,t)+L_4w(1,t)
			+\int_0^1 L_1(x)\widehat{w}(x,t)dx\\
			+\int_0^1L_2(x)\widehat{w}_t(x,t)dx\bigg]+p_x(1,t)+\alpha p(1,t)\\
			+f(w(\cdot,t),w_t(\cdot,t))+d(t), \;\;  t\ge 0, \\
			\dot{\widehat{X}}(t)=A\widehat{X}(t)+B_1\widehat{w}(0,t)+B_2w_t(0,t)+B_3w(1,t)\\
			+B_4\widehat{w}_t(1,t)+H[C\widehat{X}(t)-CX(t)],\;\;  t\ge 0, \\
			\widehat{w}_{tt}(x,t)=\widehat{w}_{xx}(x,t),\;\; x\in (0,1), \; t>0, \\
			\widehat{w}_x(0,t)=k [\widehat{w}_t(0,t)-w_t(0,t)], \; \;   t\ge 0,\\
			\widehat{w}_x(1,t)=-\alpha \widehat{w}(1,t)-\beta \widehat{w}_t(1,t)\\
			+K^T\bigg[\widehat{X}(t)+L_3\widehat{w}(0,t)
			+L_4w(1,t)\\
			+\int_0^1 L_1(x)\widehat{w}(x,t)dx+\int_0^1L_2(x)\widehat{w}_t(x,t)dx\bigg],\\
			z_{tt}(x,t)=z_{xx}(x,t),\;\; x\in (0,1), \; t>0, \\
			z_x(0,t)=k [z_t(0,t)-w_t(0,t)], \; \;   t\ge 0,\\
			z_x(1,t)=-\alpha [z(1,t)-w(1,t)]-\alpha w(1,t)-\beta \widehat{w}_t(1,t)\\
			+\alpha p(1,t)+K^T\bigg[\widehat{X}(t)+L_3\widehat{w}(0,t)
			+L_4w(1,t)\\
			+\int_0^1 L_1(x)\widehat{w}(x,t)dx+\int_0^1L_2(x)\widehat{w}_t(x,t)dx\bigg]+p_x(1,t),\\
			p_{tt}(x,t)=p_{xx}(x,t),\;\; x\in (0,1), \; t>0, \\
			p_x(0,t)=-kp_t(0,t), p(1,t)=z(1,t)-w(1,t) \;\;  t\ge 0, \\
		\end{array}\right.
	\end{align}
	
	Observe that the transformation $I+\mathbb{P}$ defined in (\ref{IjiaP}) and (\ref{PP}) is invertible with inverse $I-\mathbb{P}$ on $\mathds{R}^n\times \mathbb{H}_1$. We denote $\big(\widehat{Y}(t),\widehat{w}(\cdot,t),\widehat{w}_t(\cdot,t)\big)=(I+\mathbb{P})\big(\widehat{X}(t),\widehat{w}(\cdot,t),\widehat{w}_t(\cdot,t)\big)$,
	thereby
	\begin{align}\label{XdaoY}
		\nonumber&\widehat{Y}(t)=\widehat{X}(t)+L_3\widehat{w}(0,t)+L_4\widehat{w}(1,t)\\
		&+\int_0^1 L_1(x)\widehat{w}(x,t)dx+\int_0^1L_2(x)\widehat{w}_t(x,t)dx.
	\end{align}
	We derive the $\widehat{Y}(t)$ and $\widehat{w}(x,t)$ cascaded system as follows
	\begin{align}\label{Yjianw11}
		\left\{
		\begin{array}{ll}
			\dot{\widehat{Y}}(t)=[A+L_2(1)K^T]\widehat{Y}(t)-B_2\widetilde{w}_t(0,t)\\
			-[L_2(1)K^TL_4+B_3]\widetilde{w}(1,t)+HC\widetilde{X}(t),\\
			\widehat{w}_{tt}(x,t)=\widehat{w}_{xx}(x,t),\\
			\widehat{w}_x(0,t)=k\widetilde{w}_t(0,t), \\
			\widehat{w}_x(1,t)=-\alpha \widehat{w}(1,t)-\beta \widehat{w}_t(1,t)\\
			+K^T\widehat{Y}(t)-K^TL_4\widetilde{w}(1,t).
		\end{array}
		\right.
	\end{align}
	\begin{theorem}\label{Yjianwjian}
		The system (\ref{Yjianw11}) is exponentially stable.
	\end{theorem}
	\begin{proof}\ \
		In view of the transformation (\ref{epsilonwan}) linking $\widehat{w}(x,t)$ to $\widetilde{\epsilon}(x,t)$, we consider the energy functionals $E(t)$ and $E_2(t)$ for $\widetilde{\epsilon}$ (satisfying (\ref{epsilonwan})) and $\widehat{p}$ (satisfying (\ref{perror})), defined respectively as $E(t) = \int_0^1 [|\widetilde{\epsilon}_x(x,t)|^2 + |\widetilde{\epsilon}_t(x,t)|^2] dx + \alpha |\widetilde{\epsilon}(1,t)|^2$ and $E_2(t) = \int_0^1 [|\widehat{p}_x(x,t)|^2 + |\widehat{p}_t(x,t)|^2] dx$. Then, for any $\gamma < \min{\omega_\mathbb{A}, \omega_{\mathbb{A}_2}}$, we obtain the time derivatives $\frac{d}{dt}[e^{2\gamma t}E(t)]=2\gamma e^{2\gamma t}E(t)-ke^{2\gamma t}|\widetilde{\epsilon}_t(0,t)|^2$ and
		$\frac{d}{dt}[e^{2\gamma t}E_2(t)]=2\gamma e^{2\gamma t}E_2(t)-k e^{2\gamma t}|\widehat{p}_t(0,t)|^2.$ This leads to
		\begin{align*}
			& k\int_0^{\infty}e^{2\gamma t}|\widetilde{\epsilon}_t(0,t)|^2dt= 2\gamma \int_0^\infty e^{2\gamma t}E(t)dt+E(0)\\
			&\leq 2\gamma M^2_{\mathbb{A}} \int_0^\infty e^{2\gamma t-2\omega_\mathbb{A}t}E(t)dt+E(0)\leq \frac{M^2_{\mathbb{A}}}{\omega_\mathbb{A}-\gamma}+E(0).
		\end{align*}
		and $k\int_0^{\infty}e^{2\gamma t}|\widehat{p}_t(0,t)|^2dt\leq \frac{M^2{\mathbb{A}_2}}{\omega_{\mathbb{A}_2}-\gamma}+E(0).$
		Consequently, $e^{\gamma t}\widetilde{w}_t(0,t) = e^{\gamma t}\widetilde{\epsilon}_t(0,t) - e^{\gamma t}\widehat{p}_t(0,t) \in L^2(0,\infty)$. According to Theorem \ref{wwan}, $|\widetilde{w}(1,t)|^2 \leq \frac{1}{\alpha} |(\widetilde{w}(\cdot,t), \widetilde{w}_t(\cdot,t))|^2_{\mathbb{H}_1} \leq \frac{M_1^2}{\alpha} e^{-2\gamma_1 t}$, and $\widetilde{X}(t)$ is exponentially stable. It then follows from \cite[Lemma 2.1]{Zhou2018b} that $\widehat{Y}(t)$ is exponentially stable. Moreover, since $\mathbb{B}_{11}$ and $\delta(x)$ are admissible for $e^{\mathbb{A}_1t}$ as shown in \cite{Guo2008}, the exponential stability of $(\widehat{w}(\cdot,t), \widehat{w}_t(\cdot,t))$ is also obtained via \cite[Lemma 2.1]{Zhou2018b}, which concludes the proof.
	\end{proof}
	
	\begin{theorem}\label{main}
		Suppose that $\sigma(A)\bigcap \sigma(\mathbb{A})=\varnothing$. If (i) or (ii) in Theorem \ref{iandii} holds and we choose $K$ such that $A+L_2(1)K^T$ is a Hurwitz matrix. Assume that $f:\mathbb{H}_1\rightarrow \mathds{R}$ is continuous.
		Suppose that the initial condition satisfies $\big(X(0),w(\cdot,0),w_t(\cdot,0),\widehat{X}(0),\widehat{w}(\cdot,0),\widehat{w}_t(\cdot,0),$
		$z(\cdot,0),z_t(\cdot,0),$
		$p(\cdot,0), p_t(\cdot,0)\big)\in (\mathds{R}^n\times \mathbb{H}_1)^2\times \mathbb{H}_1^2$ with $p(1,0)=$
		$z(1,0)-w(1,0)$.
		Then, there exists unique solution $\big(X(t),w(\cdot,t),w_t(\cdot,t),\widehat{X}(t),\widehat{w}(\cdot,t),\widehat{w}_t(\cdot,t),z(\cdot,t),z_t(\cdot,t),p(\cdot,t), $
		$p_t(\cdot,t)\big)\in C\big(0,\infty;(\mathds{R}^n\times \mathbb{H}_1)^2\times \mathbb{H}_1^2\big)$.
		Moreover, there exists two constants $\gamma_2,M_2>0$ such that
		\begin{align}\label{XXjianyoujie}
			\nonumber&\big\|\big(X(t),w(\cdot,t),w_t(\cdot,t),
			\widehat{X}(t),\widehat{w}(\cdot,t),\widehat{w}_t(\cdot,t)\big)\big\|_{(\mathds{R}^n\times \mathbb{H}_1)^2}\\
			&\leq M_2e^{-\gamma_2 t}, t\geq 0
		\end{align}
		and $\sup_{t\geq 0}\|(z(\cdot,t),z_t(\cdot,t),p(\cdot,t),p_t(\cdot,t))\|_{\mathbb{H}_1^2}<\infty.$
		
		If, in addition, $d(t)=0$ and $f(0,0)=0$, then $\lim_{t\rightarrow \infty}\|(z(\cdot,t),z_t(\cdot,t),p(\cdot,t),p_t(\cdot,t))\|_{\mathbb{H}_1^2}$
		$=0$.
	\end{theorem}
	
	\begin{proof}\ \
		By Theorems \ref{Yjianwjian} and \ref{wwan}, we conclude that the system
		$(\widetilde{X}(t),\widetilde{w}(\cdot,t),\widetilde{w}_t(\cdot,t),\widehat{Y}(t),$
		$\widehat{w}(\cdot,t),\widehat{w}_t(\cdot,t))$
		uniquely exists in $C(0,\infty;(\mathds{R}^n\times\mathbb{H}_1)^2)$, is continuous, and exhibits exponential stability.
		Since $I-\mathbb{P}$ is invertible, it follows that $\widehat{X}(t)$ also uniquely exists in $C(0,\infty;\mathds{R}^n)$ and is exponentially stable.
		Consequently, $(X(t),w(\cdot,t),w_t(\cdot,t))$ uniquely belongs to $C(0,\infty;\mathds{R}^n\times \mathbb{H}_1)$ and is exponentially stable.
		Combining this with the continuity of $f$, we deduce that $f(w(\cdot,t),w_t(\cdot,t))$ is bounded.
		Then, applying \cite[Lemma 2.1]{Zhou2018b}, the admissibility of $\mathbb{B}$ for $e^{\mathbb{A}t}$, and the $\widehat{z}$-component of system (\ref{perror}),
		we establish the existence, uniqueness, continuity, and exponential stability of $(\widehat{z}(\cdot,t),\widehat{z}_t(\cdot,t))$.
		Together with the $\widehat{p}$-component of (\ref{perror}), this implies that
		$(z(\cdot,t),z_t(\cdot,t))$ and $(p(\cdot,t),p_t(\cdot,t))$ also uniquely exist, are continuous, and are exponentially stable.
		
		If, in addition, $d(t) = 0$ and $f(0,0) = 0$, then $f(w(\cdot,t),w_t(\cdot,t)) \to 0$ as $t \to \infty$,
		which implies $(\widehat{z}(\cdot,t),\widehat{z}_t(\cdot,t)) \to 0$.
		Therefore,
		$$\lim_{t\rightarrow \infty}\|(z(\cdot,t),z_t(\cdot,t),p(\cdot,t),p_t(\cdot,t))\|_{\mathbb{H}_1^2}=0,$$
		which completes the proof.
	\end{proof}

	\begin{remark}\rm
		For the ODE-wave equation system subject to boundary disturbances, the studies cited as \cite{Liu2017} and \cite{Zhou2016a} only account for external disturbances, while internal uncertainties are not considered. In contrast, our work addresses both internal uncertainties and external disturbances.
		Additionally, while \cite{Liu2017} and \cite{Zhou2016a} incorporate only a single interconnection term $w(0,t)$, our formulation includes four interconnections: $w(0,t), w_t(0,t), w(1,t)$, and $w_t(1,t)$.
		In terms of methodology, \cite{Liu2017} employs sliding mode control (SMC) with full-state feedback, whereas \cite{Zhou2016a} utilizes a sliding observer and requires four output measurements: $w_t(0,t), w(1,t), w_t(0,t)$, and $CX(t)$. However, the resulting state estimation is discontinuous due to the nature of the sliding observer.
		As established in Theorem \ref{main}, only three output measurements $w_t(0,t), w(1,t),$ and $CX(t)$ suffice for the design of the disturbance estimator, state observer, and controller in our system (\ref{beem}).
	\end{remark}

\section{Concluding remarks}
In this paper, the exponential stabilization problem via output feedback is investigated for a cascade system composed of an ODE and a wave PDE with four interconnection points. In the absence of external disturbances, a novel transformation is introduced to incorporate the boundary control input into the ODE part. Based on this transformation, a state feedback controller is designed to achieve exponential stability of the system. When both internal uncertainties and external disturbances are considered, an infinite-dimensional disturbance estimator without high-gain is constructed to estimate the total disturbance. Using this estimator, a state observer and an observer-based output feedback controller are developed. It is rigorously shown that the original system and the state observer are exponentially stable, while the overall closed-loop system remains bounded.

\end{document}